\documentclass[11pt]{article}
\usepackage[OT1, T1]{fontenc}
\usepackage{latexsym}
\usepackage{amssymb}
\usepackage{amsmath}
\usepackage{amsfonts}
\usepackage{textcomp}
\usepackage{amscd}
\usepackage{amsthm}
\usepackage{amsxtra}
\usepackage{stmaryrd}
\usepackage{fullpage}

\numberwithin{equation}{section}
\newtheorem{theorem}[equation]{Theorem}

\newtheorem{corollary}[equation]{Corollary}
\newtheorem{lemma}[equation]{Lemma}
\newtheorem{proposition}[equation]{Proposition}

\theoremstyle{definition}

\newtheorem{definition}[equation]{Definition}

\newtheorem{remark}[equation]{Remark}

\def\Z{{\mathbb {Z}}_p}
\def\Q{{\mathbb {Q}}_p}

\def\D{\mathrm{D}}

\def\ra{\rightarrow}
\def\ZZ{{\mathbb Z}}

\def\Qp{{\mathbb{Q}}_p}
\def\Cp{{\mathbb{C}}_p}
\def\Zp{{\mathbb{Z}}_p}

\def\OO{\mathcal{O}}

\def\aa{\mathbf{A}}
\def\bb{\mathbf{B}}
\def\ee{\mathbf{E}}
\def\int{\mathrm{int}}

\def\at{\widetilde{\mathbf{A}}}
\def\bt{\widetilde{\mathbf{B}}}
\def\et{\widetilde{\mathbf{E}}}

\def\etplus{\widetilde{\mathbf{E}}^+}

\def\calE{\mathcal{E}}
\def\calR{\mathcal{R}}
\def\Gal{\mathrm{Gal}}
\def\rig{\mathrm{rig}}

\def\D{\mathrm{D}}
\def\A{\mathcal{A}}
\DeclareMathOperator{\rank}{rank}

\def\m{(\varphi,\Gamma)}

\def\r{\mathcal{R}}
\newcommand{\btdag}[1]{\widetilde{\mathbf{B}}^{\dagger #1}}
\newcommand{\atdag}[1]{\widetilde{\mathbf{A}}^{\dagger #1}}
\newcommand{\bdag}[1]{\mathbf{B}^{\dagger #1}}
\newcommand{\adag}[1]{\mathbf{A}^{\dagger #1}}

\begin{document}
\title{On Families of $\m$-modules}
\author{Kiran Kedlaya\\Massachusetts Institute of Technology\\ kedlaya@mit.edu\\  \\
Ruochuan Liu\\ University of Michigan, Ann Arbor\\ ruochuan@umich.edu}

\maketitle
\begin{abstract}
Berger and Colmez introduced a theory of families of overconvergent \'etale $\m$-modules associated to families of $p$-adic Galois representations over $p$-adic Banach algebras. However, in contrast with the classical theory of $\m$-modules, the functor they obtain is not an equivalence of categories. In this paper, we prove that when the base is an affinoid space, every family of (overconvergent) \'etale $\m$-modules can locally be converted into a family of $p$-adic representations in a unique manner, providing the ``local'' equivalence. There
is a global mod $p$ obstruction related to the moduli of residual representations.
\end{abstract}

\maketitle

\section*{Introduction}
In \cite{BC07}, Berger and Colmez introduced a theory of families of overconvergent \'etale $\m$-modules associated to families of $p$-adic Galois representations over $p$-adic Banach algebras. The $p$-adic families of local Galois representations emerging from number theory are usually over rigid analytic spaces.  So we are mainly interested in the case where the bases are reduced affinoid spaces. However, even in this case the functor of Berger-Colmez is far from an equivalence of categories, in contrast with the classical theory of $\m$-modules. This was first noticed by Chenevier
\cite[Remarque~4.2.10]{BC07}: if the base is the $p$-adic unit circle $M(\Q \langle X,Y\rangle/(XY-1))$, then it is easy to see that the free rank $1$ overconvergent \'etale $\m$-module $D$ with a basis $e$ such that $\varphi(e)=Ye$ and $\gamma(e)=e$ for $\gamma\in\Gamma$ does not come from a family of $p$-adic representations over the same base.

On the other hand, in his proof of the density of crystalline representations, Colmez proved \cite[Proposition 5.2]{C08} that for certain families of $\rank$ $2$ triangular \'etale $\m$-modules, one can locally convert
such a family into a family of $p$-adic representations using his theory of \emph{Espaces Vectoriels de dimension finie} (it is clear that we can also convert Chenevier's example locally).
Moreover, Colmez remarked \cite[Remarque~5.3(2)]{C08} that: \emph{On aurait pu aussi utiliser une version \guillemotleft en famille\guillemotright\ des th\'eor\`emes \`a la Dieudonn\'e-Manin de Kedlaya. Il y a d'ailleurs une concordance assez frappante entre ce que permettent de d\'emontrer ces th\'eor\`emes de Kedlaya et la th\'eorie des Espaces Vectoriels de dimension finie}.

Unfortunately, as noticed in \cite{R08}, there is no family version of Kedlaya's slope filtrations theorem in general, because the slope polygons of families of Frobenius modules are not necessarily locally constant.
Nonetheless, one may still ask to what extent one can convert a globally \'etale
family of $\m$-modules back into a Galois representation. As Chenevier's
example shows, this cannot be done in general over an affinoid base. The best
one can hope for in general is the following theorem, which extends a result of
Dee \cite{D01}. (In the statement, the distinction between a
$\m$-module and a family of $\m$-modules is that the former is defined
as a module over a ring, whereas the latter is defined as a coherent
sheaf over a rigid analytic space.)

\begin{theorem}\label{thm:A}
Let $S$ be a Banach algebra over $\Qp$ of the form $R \otimes_{\Zp} \Qp$,
where $R$ is a complete noetherian local domain of characteristic $0$
whose residue field is finite over $\mathbb{F}_p$. Then for any finite extension $K$ of
$\Qp$, the categories of $S$-linear representations of $G_K$, of
\'etale $\m$-modules over $\bb^\dagger_K \widehat{\otimes}_{\Qp} S$, and of
families of
\'etale $\m$-modules over $\bb^\dagger_{\rig,K} \widehat{\otimes}_{\Qp} S$
are all equivalent.
\end{theorem}
For instance, if $S$ is an affinoid algebra and we are given an
\'etale $\m$-module
over $\bb^\dagger_{K} \widehat{\otimes}_{\Qp} S$, we recover a linear
representation over each residue disc of $S$ (and every affinoid subdomain of
such a disc), but these representations may not glue. This is what happens
in Chenevier's example, because the mod $p$ representations cannot be
uniformly
trivialized. In fact, the obstruction to converting a $\m$-module back
into a representation exists purely at the residual level;
it suggests a concrete realization of the somewhat
murky notion of ``moduli of residual (local) representations''.

By combining Theorem~\label{thm:A} with the results of \cite{R08}, we obtain a result
that applies when only one fibre of the $\m$-module is known to be \'etale.
(Beware that the natural analogue of this statement in which the
rigid analytic point $x$ is replaced by a Berkovich point is trivially false.)
\begin{theorem}\label{thm:B}
Let $S$ be an affinoid algebra over $\Q$, and let $M_S$ be a family of $\m$-modules over $\bb^\dagger_{\rig,K}
\widehat{\otimes}_{\Qp} S$.
If $M_x$ is \'etale for some $x\in M(S)$,
then there exists an affinoid neighborhood $M(B)$ of $x$ and
a $B$-linear representation $V_B$ of $G_K$ whose associated $\m$-module
is isomorphic to $M_S \widehat{\otimes}_S B$.
Moreover, $V_B$ is unique for this property.
\end{theorem}

In \cite{LB06}, to prove the Fontaine-Colmez theorem, Berger constructed a morphism from the category of filtered $(\phi,N)$-modules to the category of $\m$-modules. It should be possible to generalize Berger's construction to families of filtered $(\phi, N)$-modules; upon doing so, one would get a family version of the Fontaine-Colmez theorem by the preceding theorem. That is, one would know
that a weakly admissible family of filtered $(\phi, N)$-modules over an affinoid base (with trivial
$\phi$-action on the base) becomes admissible in a neighborhood of
each rigid analytic point.

\subsection*{Acknowledgments}
The first author thanks Siegfried Bosch for helpful comments.
The second author thanks Pierre Colmez for the helpful conversation during
spring 2008
which inspired this project,
Laurent Berger for answering many questions patiently,
and Marie-France Vign\'eras for arranging his visit to
Institut de Math\'ematiques de Jussieu
during spring 2008. While writing this paper, the first author
was supported by NSF CAREER grant DMS-0545904,
the MIT NEC Research Support Fund,
and the MIT Cecil and Ida Green Career Development Professorship,
while the second author was a postdoctoral fellow of Foundation Sciences Math\'ematiques de Paris.

\section{Rings of $p$-adic Hodge theory}

We begin by introducing some of the rings used in $p$-adic Hodge theory.
This is solely to fix notation; we do not attempt to expose the constructions
in any detail. For that, see for instance \cite{B04}.
In what follows, whenever a ring is defined whose notation includes
a boldface $\aa$, the same notation with $\aa$ replaced by $\bb$
will indicate the result of inverting $p$.

Let $\Cp$ be a completed algebraic closure of $\Q$,
with valuation subring $\OO_{\Cp}$ and $p$-adic valuation $v_p$
normalized with $v_p(p) = 1$. Let $\overline{v}_p: \OO_{\Cp}/(p) \to
[0, 1) \cup \{+\infty\}$ be the semivaluation obtained by truncation.
Define $\etplus$ to be the ring of sequences
$(x_n)_{n=0}^\infty$ in $\OO_{\Cp}/(p)$ such that $x_{n+1}^p = x_n$
for all $n$. Define a function $v_{\ee}: \etplus
\to [0, +\infty]$ by
sending the zero sequence to $+\infty$, and sending each nonzero sequence $(x_n)$
to the common value of $p^n \overline{v}_p(x_n)$ for all $n$ with
$x_n \neq 0$. This gives a valuation under which
$\etplus$ is complete.
Moreover, if we put $\et =
\mathrm{Frac}(\etplus)$, and let
$\epsilon = (\epsilon_n)$ be an element of $\etplus$ with
$\epsilon_0 = 1$
and $\epsilon_1 \neq 1$, then $\et$ is a completed algebraic closure of
$\mathbb{F}_p((\epsilon - 1))$.

Let $\at$ be the $p$-typical Witt ring $W(\et)$,
which is the unique complete discrete valuation ring with maximal ideal $(p)$
and residue field $\et$.
For each positive integer $n$, $W(\et)/p^n W(\et)$ inherits a
topology from the valuation topology on $\et$, under which it is complete.
We call the inverse limit
of these the \emph{weak topology} on $\at$.
We similarly obtain a weak topology on $\bt$.

For any $n\geq0$, we let $\mu_{p^n}$ denote the set of $p^n$-th roots of unity in $\overline{\mathbb{Q}}_p$, and let $\mu_{p^\infty}=\cup_{n\geq0}\mu_{p^n}$. For $K$ a finite extension of $\Q$, let $K_\infty=K(\mu_{p^\infty}), H_K = \Gal(\overline{K}/K_\infty), \Gamma =\Gamma_K = \Gal(K_\infty/K)$ and $K_0'=\mathbb{Q}_p^{\mathrm{ur}}\cap K_\infty$.

Put $\pi = [\epsilon] - 1$, where brackets denote the Teichm\"uller lift.
Using the completeness of $\at$ for the weak topology, we may embed
$\Zp((\pi))$ into $\at$. Let $\aa$ be the $p$-adic completion of the
 integral closure of $\Zp((\pi))$ in $\at$, and put $\aa_K = \aa^{H_K}$.
These rings carry actions of $G_{K}$
which are continuous for the weak topology
on the rings and the profinite topology on $G_{K}$.
They also carry endomorphisms
 $\varphi$ (which are weakly and $p$-adically continuous)
induced by the Witt vector Frobenius on
$\at$.

For $s > 0$, the subset
\[
\at^{\dagger,s} = \{x \in \at: x = \sum_{k \in \ZZ} p^k [x_k],
v_{\et}(x_k) + \frac{psk}{p-1} \geq 0,
\lim_{k \to +\infty} v_{\et}(x_k) + \frac{psk}{p-1} = +\infty \}
\]
is a subring of $\at$
which is complete for the valuation
\[
w_s(x) = \inf_k \left\{v_{\et}(x_k) + \frac{psk}{p-1} \right\}.
\]
Put $\bt^\dagger = \cup_{s>0} \bt^{\dagger,s}$,
$\bb^{\dagger,s}_K = \bb_K \cap \bt^{\dagger,s}$,
$\bb_K^\dagger = \cup_{s>0} \bb^{\dagger,s}_K$,
$\aa_K^{\dagger,s} = \aa_K\cap \at^{\dagger,s}$, $\aa_K^\dagger = \aa \cap \bb_K^\dagger$.
(Beware that the latter ring
is strictly larger than $\cup_{s>0} \aa_K^{\dagger,s}$.)
These rings carry an action of $\varphi$; for $n$ a positive integer, write
\[
\adag{,s}_{K,n}=\varphi^{-n}(\adag{,p^n s}_K).
\]

Let $\bt^{\dagger,s}_{\rig}$
be the Fr\'echet completion of $\bt^{\dagger,s'}$
under the valuations $w_{s'}$ for all $s' \geq s$, and put
$\bt^{\dagger}_{\rig} = \cup_{s>0} \bt^{\dagger,s}_{\rig}$.
Similarly, let $\bb^{\dagger,s}_{\rig,K}$
be the Fr\'echet completion of $\bb_K^{\dagger,s'}$
under the valuations $w_{s'}$ for all $s' \geq s$,
and put $\bb^{\dagger}_{\rig,K} = \cup_{s>0} \bb^{\dagger,s}_{\rig,K}$.
It turns out that $(\bb^{\dagger,s}_{\rig})^{H_K} = \bb^{\dagger,s}_{\rig,K}$.

Some of these rings admit more explicit descriptions, as follows.
It turns out that $\bb_K$ is isomorphic to the
$p$-adic local field
\[
\calE_{K_0'}=\{f=\sum^{+\infty}_{i=-\infty}a_i T^i\mid a_i\in K_0',
\inf_i \{v_p(a_i)\} > -\infty,
\lim_{i \to -\infty} v_p(a_i) = +\infty\}
\]
with valuation $w(f)=\min_{i\in\mathbb{Z}}v_p(a_i)$ and imperfect
residue field $k'((T))$, where $k'$ is the residue field of $K_0'$.
There is no distinguished such isomorphism in general
(except for $K = \Qp$, where one may take $T = \pi$), but suppose we
fix a choice. Then $\bb$ corresponds to
the completion of the maximal unramified
extension of $\bb_K$.
For $s \gg 0$ (depending on $K$ and the choice of the isomorphism
$\bb_K \cong \calE_{K_0'}$), $\bb^{\dagger,s}_{K}$ corresponds to the subring
$\calE^s_{K_0'}$ of $\calE_{K'_0}$ defined as
\[
\calE^s_{K_0'} =\{f=\sum^{+\infty}_{i=-\infty}a_i T^i\mid a_i\in K_0',
\inf_i \{v_p(a_i)\} > -\infty, \lim_{i \to -\infty}
i + \frac{ps}{p-1} v_p(a_i) = +\infty\},
\]
i.e., the bounded  Laurent series in $T$ convergent
on the annulus $0 < v_p(T) \leq 1/s$.
Meanwhile, $\bb_{\rig,K}^{\dagger,s}$
corresponds to the ring
\[
\calR^s_{K_0'} =\{f=\sum^{+\infty}_{i=-\infty}a_i T^i\mid a_i\in K_0',
\lim_{i \to +\infty}
i + r v_p(a_i) = +\infty \,\, \forall r>0,
\lim_{i \to -\infty}
i + \frac{ps}{p-1} v_p(a_i) = +\infty\},
\]
i.e., the unbounded  Laurent series in $T$ convergent
on the annulus $0 < v_p(T) \leq 1/s$.
The union $\calR_{K_0'} = \cup_{s>0} \calR^s_{K_0'}$
is commonly called the \emph{Robba ring}
over $K'_0$.

\section{$p$-adic representations and $\m$-modules}

We next introduce $p$-adic representations and the objects
of semilinear algebra used to describe them.
Fix a finite extension $K$ of $\Qp$.
For $R$ a topological ring, we will mean by an \emph{$R$-linear representation}
a finite $R$-module equipped with a continuous linear action of $G_K$.
(We will apply additional adjectives like ``free'', which are to be
passed through to the underlying $R$-module.)
Fontaine \cite{F91} constructed a functor giving an equivalence of
categories between $\Q$-linear representations and certain linear (or rather
semilinear) algebraic data, as follows. (We may extend to $L$-linear
representations for finite extensions $L$ of $\Q$,
by restricting the coefficient field to $\Q$ and then keeping
track of the $L$-action separately.)

An \emph{\'etale $\varphi$-module} over $\aa_K$ is a finite
module $N$ over $\aa_K$,
equipped with a semilinear action of $\varphi$, such that the
induced $\aa_K$-linear map $\varphi^* N \to N$ induced by the $\varphi$-action
is an isomorphism.
An \emph{\'etale $\varphi$-module} over $\bb_K$ is a finite
module $M$ over $\bb_K$,
equipped with a semilinear action of $\varphi$, which contains
an $\aa_K$-lattice $N$ (i.e., a finite $\aa_K$-submodule such that the induced
map $N \otimes_{\aa_K} \bb_K \to M$ is an isomorphism)
which forms an \'etale $\varphi$-module over $\aa_K$.
An \emph{\'etale $\m$-module} over $\aa_K$ or $\bb_K$
is an \'etale $\varphi$-module
equipped with a semilinear action of $\Gamma$ which commutes with the
$\varphi$-action and is continuous for the profinite topology on
$\Gamma$ and the weak topology on $\aa_K$.
Note that an \'etale $\m$-module over $\bb_K$ may contain
an $\aa_K$-lattice which forms an \'etale $\varphi$-module over $\aa_K$
but is not stable under $\Gamma$;
on the other hand,
the images of such a lattice under $\Gamma$ span another lattice
which forms an \'etale $\varphi$-module over $\aa_K$.

For $T$ a $\Z$-linear representation,
define $\D(T)=(\aa\otimes_{\Zp} T)^{H_K}$; this gives
an $\aa_K$-module equipped with commuting semilinear actions
of $\varphi$ and $\Gamma$. Similarly,
for $V$ a $\Q$-linear representation,
define $\D(V) = (\bb \otimes_{\Qp} V)^{H_K}$.

\begin{theorem}[Fontaine] \label{thm:Fontaine}
The functor $T \mapsto \D(T)$ (resp.\
$V\mapsto\D(V)$)
is an equivalence from the category of $\Z$-linear representations
(resp.\ $\Q$-linear representations) of $G_{K}$ to the category of \'etale $(\varphi,\Gamma)$-modules over $\aa_{K}$ (resp.\ $\bb_K$); a quasi-inverse functor is given by $D\mapsto(\aa\otimes_{\aa_K} D)^{\varphi=1}$ (resp.\ $D\mapsto (\bb\otimes_{\bb_K} D)^{\varphi=1}$).
\end{theorem}
Dee \cite{D01} extended Fontaine's results to families of $\Z$-representations,
as follows.
Let $R$ be a complete noetherian local ring whose residue field $k_R$ is finite over $\mathbb{F}_p$, equipped with the topology defined by its maximal
ideal $\mathfrak{m}_R$; we may then view $R$ as a topological $\Z$-algebra.
We form the completed tensor product $R\widehat{\otimes}_{\Z}\aa$
by completing the ordinary tensor product for
the ideal $p \aa + \mathfrak{m}_R$, and similarly with $\aa$ replaced by
$\aa_K$.

We define $\m$-modules and \'etale $\m$-modules over $R\widehat{\otimes}_{\Z} \aa_K $
by analogy with the definitions over $\aa_K$.
For $T_R$ an $R$-representation, define
$\D(T_R) =((R\widehat{\otimes}_{\Z}\aa)\otimes_R T_R)^{H_K}$.
We then have the following result.

\begin{theorem}[Dee] \label{thm:Dee}
The functor $T_R\mapsto\D(T_R)$ is an equivalence from the category of
$R$-representations to the category of \'etale
$\m$-modules over $R\widehat{\otimes}_{\Z}\aa_{K}$;
a quasi-inverse functor is given by
\begin{center}
$D\mapsto((R\widehat{\otimes}_{\Z}\aa)\otimes_{R\widehat{\otimes}_{\Z}\aa_K} D)^{\varphi=1}$.
\end{center}
\end{theorem}

We next introduce a refinement of Fontaine's result due to
Cherbonnier and Colmez \cite{CC98}.
We define $\m$-modules and \'etale $\m$-modules over the rings
$\aa^\dagger_K$
and
$\bb^\dagger_K$  by analogy with the definitions
over $\aa_K$ and $\bb_K$.
For $V$ a $\Qp$-linear representation, define
$\D_K^{\dagger,r}(V)=(\bdag{,r}\otimes_{\Qp} V)^{H_{K}}$ (where $\bb^{\dagger,r}=\bb\cap\bt^{\dagger,r}$) and
$\D_K^{\dagger}(V)=\cup_{r> 0} \D_K^{\dagger,r}(V)=(\bdag{}\otimes_{\Qp} V)^{H_K}$.
\begin{theorem}[Cherbonnier-Colmez] \label{thm:CC}
For each $\Qp$-linear representation $V$, there exists $r(V) > 0$
such that
\[
\D_K(V)=\bb_{K}\otimes_{\bdag{,r}_{K}}\D^{\dagger,r}_K(V) \qquad
\mbox{for all } r \geqslant r(V).
\]
Equivalently, $\D^{\dagger}_K(V)$ is an \'etale $\m$-module over
$\bb_K^\dagger$ of dimension $\dim_{\Q}V$.
Therefore $V\mapsto\D^\dagger_K(V)$ is an equivalence from the category of $p$-adic
representations of $G_{K}$ to the category of \'etale $(\varphi,\Gamma)$-modules over $\bb_K^\dagger$.
Furthermore, $\D^{\dagger}_K(V)$ is the unique maximal \'etale $\m$-submodule of $\D_K(V)$ over $\bb_K^\dagger$.
\end{theorem}

In \cite{BC07}, Berger and Colmez extended these results to families of
$p$-adic representations. However, unlike Dee's families,
the families considered by Berger and Colmez are over
Banach algebras over $\Q$.
(Berger and Colmez are forced to make a freeness hypothesis on the
representation space; we will relax this hypothesis later in the case
of an affinoid algebra. See Definition~\ref{def:fm-fg}.)

For $S$ a commutative Banach algebra over $\Qp$,
let $\OO_S$ be the ring of elements of $S$ of norm at most 1,
and let $I_S$ be the ideal of elements of $\OO_S$ of norm  strictly
less than 1.
Note that it makes sense to form a completed tensor product with
$S$ or $\OO_S$ when the other tensorand carries a norm under it is complete,
e.g.,
for the rings $\atdag{,s}, \adag{,s}_{L,n}, \btdag{,s}, \bdag{,s}_L$
using the norm corresponding to the valuation $w_s$.

\begin{proposition}[{\cite[Proposition 4.2.8]{BC07}}]\label{prop:BC}
Let $S$ be a commutative Banach algebra over $\Qp$.
Let $T_S$ be a free $\OO_S$-linear representation of rank $d$.
Let $L$ be a finite Galois extension of $K$ such that
$G_L$ acts trivially on $T_S/12pT_S$.
Then there exists $n(L,T_S) \geq 0$ such that
for $n\geq n(L,T_S)$,
$(\OO_S\widehat{\otimes}_{\Zp} \tilde{\aa}^{\dag,(p-1)/p})\otimes_{\OO_S}T_S$ has a
 unique sub-$(\OO_S\widehat{\otimes}_{\Zp} \aa^{\dag,(p-1)/p}_{L,n})$-module
$\D^{\dagger,(p-1)/p}_{L,n}(T_S)$ which is free of rank $d$,
is fixed by $H_L$,
has a basis which is almost invariant under $\Gamma_L$
(i.e., for each $\gamma \in \Gamma_L$, the matrix of action
of $\gamma-1$ on the basis has positive valuation),
and satisfies
\[
(\OO_S\widehat{\otimes}_{\Zp} \tilde{\aa}^{\dag,(p-1)/p})\otimes_{\OO_S\widehat{\otimes}_{\Zp} \aa^{\dag,(p-1)/p}_{L,n}}\D^{\dagger,(p-1)/p}_{L,n}(T_S)=
(\OO_S\widehat{\otimes}\tilde{\aa}^{\dag,(p-1)/p})\otimes_{\OO_S}T_S.
\]\end{proposition}

\begin{theorem}[{\cite[Th\'eor\`eme~4.2.9]{BC07}}]\label{thm:BC}
Let $S$ be a commutative Banach algebra over $\Qp$.
Let $V_S$ be an $S$-linear representation
admitting a free Galois-stable $\OO_S$-lattice $T_S$. There exists
an $s(V_S)\geq 0$ such that for any $s\geq s(V_S)$, we may define
\begin{center}
$\D^{\dagger,s}_K(V_S)=((S\widehat{\otimes}_{\Qp} \bb^{\dag,s}_L)\otimes_{\OO_S\widehat{\otimes}_{\Zp} \aa^{\dag,s(V_S)}_L}\varphi^{n} (\D^{\dagger, p-1/p}_{L,n}(T_S)))^{H_K}$
\end{center}
for some $L,n$, so that the construction does not depend
on the choices of $T_S, L, n$, and the following statements hold.
\begin{enumerate}
\item[(1)]
The $(S\widehat{\otimes}_{\Qp} \bb_K^{\dagger,s})$-module
$\D^{\dagger,s}_K(V_S)$ is locally free of rank $d$.
\item[(2)]The natural map $\D^{\dagger,s}_K(V_S)\otimes_{S\widehat{\otimes}_{\Qp} \mathbf{B}_K^{\dagger,s}} (S\widehat{\otimes}_{\Qp} \widetilde{\mathbf{B}}^{\dagger,s}) \ra V_S\otimes_S (S\widehat{\otimes}_{\Qp} \widetilde{\mathbf{B}}^{\dagger,s})$ is an isomorphism.
\item[(3)]For any maximal ideal $\mathfrak{m}_x$ of $S$,
for $V_x = V_S \otimes_S (S/\mathfrak{m}_x)$,
the natural map $\D^{\dagger,s}_K(V_S) \otimes_S (S/\mathfrak{m}_x)
\ra\D^{\dagger,s}_K(V_x)$ is an isomorphism.
\end{enumerate}
\end{theorem}
We put $S\widehat{\otimes}_{\Qp} \bb^\dagger_K = \cup_{s>0}
(S\widehat{\otimes}_{\Qp} \bb^{\dagger,s}_K)$ and
$S\widehat{\otimes}_{\Qp} \bt^\dagger = \cup_{s>0}
(S\widehat{\otimes}_{\Qp} \bt^{\dagger,s})$.
(Beware that $S\widehat{\otimes}_{\Qp} \bb^\dagger_K$ does not necessarily
embed into $S\widehat{\otimes}_{\Qp} \bb_K$, due to the incompatibility between
the topologies used for the completed tensor products.)
We then put
\[
\D_K^\dagger(V_S)=\D^{\dagger,s}_K(V_S)\otimes_{S\widehat{\otimes}_{\Qp} \bdag{,s}_K } (S\widehat{\otimes}_{\Qp} \bb^\dagger_K).
\]
We may recover $V_S$ from $\D_K^\dagger(V_S)$ as follows.

\begin{lemma}\label{lem:phi=1}
We have $(S\widehat{\otimes}_{\Qp} \bt^\dagger)^{\varphi=1}= S$.
\end{lemma}
\begin{proof}
We reduce at once to the case where $S$
is countably topologically generated over $\Qp$. In this case,
by \cite[Proposition~2.7.2/3]{BGR84},
we can find a \emph{Schauder basis} of $S$ over $\Q$; in other words,
there exists an index set $I$ such that $S$ is isomorphic as a topological
$\Q$-vector space to the Banach space
\begin{center}
$l^\infty_0(I,\Q)=\{(a_i)_{i\in I}\mid a_i\in \Q, a_i\ra0\}$.
\end{center}
(The supremum norm need only be equivalent to the Banach norm on $S$;
the two need not be equal.)
We can then write $S \widehat{\otimes}_{\Qp} \bt^\dagger$, as a topological
$\Q$-vector space, as
\begin{center}
$l^\infty_0(I,\bt^\dagger)=\{(a_i)_{i\in I}\mid a_i\in \bt^\dagger, a_i\ra0\}$.
\end{center}
In this presentation, the $\varphi$-action carries
$(a_i)_{i\in I}$ to $(\varphi(a_i))_{i \in I}$. It is then clear that
$(S \widehat{\otimes}_{\Qp} \bt^\dagger)^{\varphi=1}=(l^\infty_0(I,\bt^\dagger))^{\varphi=1}=l^\infty_0(I,\Q)=S$.
\end{proof}

\begin{proposition}\label{prop:phi=1}
We have
\[
(\D_K^\dagger(V_S)\otimes_{S\widehat{\otimes}_{\Qp} \bb^\dagger_K}(S\widehat{\otimes}_{\Qp} \bt^\dagger))^{\varphi=1} = V_S.
\]
\end{proposition}
\begin{proof}
From Theorem~\ref{thm:BC}(2) we get $\D_{K}^\dagger(V_S)\otimes_{S\widehat{\otimes}_{\Qp} \bb^\dagger_{K}} (S\widehat{\otimes}_{\Qp} \bt^\dagger)
=V_S\otimes_S(S\widehat{\otimes}_{\Qp} \bt^\dagger)$.
It follows that
\begin{center}
$(\D_{K}^\dagger(V_S)\otimes_{S\widehat{\otimes}_{\Qp} \bb^\dagger_{K} }
(S\widehat{\otimes}_{\Qp} \bt^\dagger))^{\varphi=1} =V_S\otimes_S
(S\widehat{\otimes}_{\Qp} \bt^\dagger)^{\varphi=1}=V_S$
\end{center}
by Lemma~\ref{lem:phi=1}.
\end{proof}

This suggests that the object $\D_K^\dagger(V_S)$
merits the following definition.
\begin{definition}\label{def:(phi,Gamma)-module}
Define a \emph{$\m$-module} over
$S\widehat{\otimes}_{\Qp} \bb^\dagger_K$
to be a finite
locally free module over
$S\widehat{\otimes}_{\Qp} \bb^\dagger_K$,
equipped with commuting continuous $\m$-actions such that $\varphi^*D_S\ra D_S$ is an isomorphism.
We say a $\m$-module $M_S$ over $S\widehat{\otimes}_{\Qp} \bb^\dagger_{K}$
is \emph{\'etale} if it admits a finite $(\varphi, \Gamma)$-stable
$(\OO_S \widehat{\otimes}_{\Zp}\aa^\dagger_{K})$-submodule $N_S$
such that $\varphi^*N_S\ra N_S$ is an isomorphism and the induced map
\[
N_S \otimes_{\OO_S \widehat{\otimes}_{\Zp}\aa^\dagger_{K}}
S\widehat{\otimes}_{\Qp} \bb^\dagger_{K} \to M_S
\]
is an isomorphism.
In this language, Theorem~\ref{thm:BC} implies that
$\D_K^\dagger(V_S)$
is an \'etale $\m$-module over $S\widehat{\otimes}_{\Qp} \bb^\dagger_K$.
\end{definition}

\section{Gluing on affinoid spaces}

Throughout this section, let $S$ denote an affinoid algebra over $\Qp$.
We explain in this section how to perform gluing for finite modules
over $S\widehat{\otimes}_{\Qp} \bb^\dagger_K$.
We start with some basic notions from \cite{BGR84}.
\begin{definition}\label{def:affinoid}
Let $M(S)$ be the set of maximal ideals of $S$, i.e., the affinoid
space associated to $S$.
For $X$ a subset of $M(S)$,
an \emph{affinoid subdomain} of $X$ is a subset $U$ of $X$ for which
there exists a morphism $S \to S'$ of affinoid algebras such that
the induced map $M(S') \to M(S)$ is universal for maps from an affinoid
space to $M(S)$ landing in $U$. The algebra $S'$ is then unique up to unique
isomorphism, and the resulting map $M(S) \to U$ is a bijection.

The set $M(S)$ carries two canonical $G$-topologies, defined as follows.
In the \emph{weak $G$-topology}, the admissible open sets are the
affinoid subdomains, and the admissible coverings are the finite
coverings.
In the \emph{strong $G$-topology},
the admissible open sets are the subsets $U$ of $M(S)$ admitting a covering
by affinoid subdomains such that the induced covering of any
affinoid subdomain of $U$ can be refined to a finite cover
by affinoid subdomains,
and the admissible coverings are the ones whose restriction to any
affinoid subdomain can be refined to a finite cover
by affinoid subdomains.
The categories of sheaves on these two topologies are equivalent,
because the strong $G$-topology is \emph{slightly finer} than the
weak one \cite[\S 9.1]{BGR84}.
\end{definition}

We need a generalization
of the Tate and Kiehl theorems on coherent sheaves
on affinoid spaces.
\begin{definition}\label{def:presheaf}
For $A$ a commutative Banach algebra over $\Qp$, define the presheaf $\A$
on the weak $G$-topology of $M(S)$ by declaring that
\[
\A(M(S')) = S' \widehat{\otimes}_{\Qp} A.
\]
\end{definition}

\begin{lemma}\label{lem:presheaf=sheaf}
For $A$ a commutative Banach algebra over $\Qp$, the presheaf $\A$
is a sheaf for the weak $G$-topology of $M(S)$,
and hence extends uniquely to the strong $G$-topology.
\end{lemma}
\begin{proof}
Since every finite covering of an affinoid space by affinoid subdomains
can be refined to a Laurent covering, it is enough to check the sheaf
condition for Laurent coverings \cite[Proposition~8.2.2/5]{BGR84}.
This reduces to checking for coverings of the form
\[
M(S) = M(S \langle f \rangle) \cup M(S \langle f^{-1} \rangle)
\]
for $f \in S$.
The claim then is that the sequence
\[
0 \to S \widehat{\otimes}_{\Qp} A \to
(S\langle f \rangle \widehat{\otimes}_{\Qp} A)
\times (S
\langle f^{-1} \rangle \widehat{\otimes}_{\Qp} A)
\stackrel{d^0}{\to} S \langle f, f^{-1} \rangle\widehat{\otimes}_{\Qp} A
\to 0
\]
is exact; this follows from the corresponding assertion for $A = \Qp$,
for which see \cite[\S 8.2.3]{BGR84}.
\end{proof}

From now on, we consider only the strong $G$-topology on $M(S)$.
\begin{definition}\label{def:coherent}
For $A$ a commutative Banach algebra over $\Qp$,
an $\A$-module $N$ on $M(S)$
is \emph{coherent} if there exists an admissible
covering $\{M(S_i)\}_{i \in I}$ of $M(S)$
by affinoid subdomains such that for
each $i \in I$, we have $N|_{M(S_i)} = \mathrm{coker}(\phi: \A^m|_{M(S_i)}
\to \A^n|_{M(S_i)})$ for some morphism $\phi$ of $\A$-modules.
By Lemma~\ref{lem:presheaf=sheaf}, this is equivalent to requiring $N|_{M(S_i)}$ to be the
sheaf associated to some finitely presented
$(S_i \widehat{\otimes}_{\Qp} A)$-module.
\end{definition}

\begin{lemma}\label{lem:cech-vanish}
Let $A$ be a commutative Banach algebra over $\Qp$
such that for each Tate algebra $T_n$ over $\Qp$,
$T_n \widehat{\otimes}_{\Qp} A$ is noetherian.
Then for any coherent $\A$-module $N$ on $M(S)$, the first \v{C}ech cohomology
$\check{H}^1(N)$ vanishes.
\end{lemma}
\begin{proof}
As in Lemma~\ref{lem:presheaf=sheaf}, it suffices to check vanishing of the first
\v{C}ech cohomology computed on a cover of $M(S)$ of the form
\[
M(S) = M(S \langle f \rangle) \cup M(S \langle f^{-1} \rangle)
\]
for some $f \in S$, such that $N$ is represented on each of
the two covering subsets by a finite module. For this, we may follow
the proof of \cite[Lemma~9.4.3/5]{BGR84} verbatim.
(The noetherian condition is needed so that the invocation of
\cite[Proposition~3.7.3/3]{BGR84} within the proof of
\cite[Lemma~9.4.3/5]{BGR84} remains valid.)
\end{proof}

To recover an analogue of Kiehl's theorem, however, we need an
extra condition.
\begin{proposition}\label{prop:noetherian}
Let $A$ be a commutative Banach algebra over $\Qp$ such that for each
Tate algebra $T_n$ over $\Qp$,
$T_n \widehat{\otimes}_{\Qp} A$ is noetherian
and
the map $\mathrm{Spec}(T_n \widehat{\otimes}_{\Qp}
A) \to \mathrm{Spec}(T_n)$
carries $M(T_n \widehat{\otimes}_{\Qp} A)$ to $M(T_n)$.
Then any coherent $\A$-module $N$ on $M(S)$
is associated to a finite
$(S \widehat{\otimes}_{\Qp} A)$-module.
\end{proposition}
\begin{proof}
There must exist a finite covering of $M(S)$ by affinoid subdomains
$M(S_1),\dots,M(S_n)$ such that $N|_{M(S_i)}$ is associated
to a finite $(S_i \widehat{\otimes}_{\Qp} A)$-module $N_i$.
As in \cite[Lemma~9.4.3/6]{BGR84}, we may deduce from Lemma~\ref{lem:cech-vanish}
that for each $\mathfrak{m} \in M(S_i)$, the map
$N(M(S)) \to (N/\mathfrak{m} N)(M(S_i))$ is surjective.
By the hypothesis on $A$,
each maximal ideal of $S_i \widehat{\otimes}_{\Qp} A$ lies over a maximal ideal
of $S_i$; we may thus deduce that $N(M(S)) \otimes_S S_i$ surjects onto $N(M(S_i))$.
Since the latter is a finite $S_i\widehat\otimes_{\Q}A$-module, we can choose finitely many
elements of $N(M(S))$ which generate all of the $N(M(S_i))$.
That is, we have a surjection $\A^n \to N$ for some $n$;
repeating the argument for the kernel of this map yields the claim.
\end{proof}

To use the above argument, we need to prove a variant of the Nullstellensatz;
for simplicity, we restrict to the case where $K$ is discretely valued
(the case of interest in this paper).
We first prove a finite generation result using ideas from the theory
of Gr\"obner bases.
\begin{lemma}\label{lem:noetherian}
Let $K$ be a complete discretely valued field extension of $\Qp$.
Let $A$ be a commutative Banach algebra over $K$
such that $A$ has the same set of nonzero norms as $K$,
and the ring $\OO_A/I_A$ is noetherian.
Then $T_n \widehat{\otimes}_{\Qp} A$ is also noetherian.
\end{lemma}
\begin{proof}
Equip the monoid $\ZZ^n_{\geq 0}$
with the componentwise partial ordering
$\leq$ and the lexicographic total ordering $\preceq$.
That is, $(x_1,\dots,x_n) \leq (y_1,\dots,y_n)$ if $x_i \leq y_i$ for all
$i$, whereas $(x_1,\dots,x_n) \preceq (y_1,\dots,y_n)$ if there exists
an index $i \in \{1,\dots,n+1\}$ such that $x_j = y_j$ for $j< i$,
and either $i=n+1$ or $x_i \leq y_i$.
Recall that $\leq$ is a well partial ordering and that $\preceq$
is a well total ordering; in particular, any sequence in
$\ZZ^n_{\geq 0}$ has a subsequence which is weakly increasing
under both orderings.

For $I = (i_1,\dots,i_n) \in \ZZ^n_{\geq 0}$, write $t^I$ for
$t_1^{i_1} \cdots t_n^{i_n}$. We represent each element
$x \in T_n \widehat{\otimes}_{\Qp} A$
as a formal sum $\sum_I x_I t^I$ with
$x_I \in A$, such that for each $\epsilon > 0$, there exist only finitely
many indices $I$ with $|x_I| \leq \epsilon$.
For $x$ nonzero, define the \emph{degree} of $x$, denoted $\deg(x)$,
to be the maximal index $I$ under $\preceq$ among those indices maximizing
$|x_I|$. Define the \emph{leading coefficient} of $x$ to be the coefficient
$x_{\deg(x)}$.

Let $J$ be any ideal of $T_n \widehat{\otimes}_{\Qp} A$.
We apply a Buchberger-type algorithm to construct a generating set
$m_1,\dots,m_k$ for $J$, as follows.
Start with the empty list (i.e., $k=0$).
As long as possible, given $m_1,\dots,m_k$,
choose an element $m_{k+1}$ of $J \cap \OO_A$
with leading coefficient $a_{k+1}$,
for which we \emph{cannot} choose
$I_1,\dots,I_k \in \ZZ_{\geq 0}^n$ and $b_1,\dots,b_k \in \OO_A$ satisfying
both of the following conditions.
\begin{enumerate}
\item[(a)]
We have $\deg(m_{k+1}) = \deg(m_i t^{I_i})$ whenever $b_i \neq 0$.
\item[(b)]
We have $a_{k+1} - a_1 b_1 - \cdots - a_k b_k \in I_A$.
\end{enumerate}
In particular, we must have $|a_{k+1}| = 1$.

We claim this process must terminate. Suppose the contrary;
then there must exist a sequence of indices $i_1 < i_2 < \cdots$ such that
$\deg(m_{i_1}) \leq \deg(m_{i_2}) \leq \cdots$.
Then the sequence of ideals
$(a_{i_1}), (a_{i_1}, a_{i_2}), \dots$ in $\OO_A/I_A$ must be strictly increasing,
but this violates the hypothesis that
$\OO_A/I_A$ is noetherian. Hence the process terminates.

Let $|\cdot|_1$ denote the 1-Gauss norm on $T_n \widehat{\otimes}_{\Qp} A$.
We now write each element of $J$ as a linear combination of $m_1,\dots,m_k$
using a form of the Buchberger division algorithm.
Start with some nonzero $x \in J$ and put $x_0 = x$.
Given $x_l \in J$,
if $x_l = 0$, put $y_{l,1} = \cdots = y_{l,k} = 0$ and $x_{l+1} = 0$.
Otherwise, choose $\lambda \in A^\times$ with $|\lambda x_l|_1 = 1$.
By the construction of $m_1,\dots,m_k$, there must exist
$I_1,\dots,I_k \in \ZZ_{\geq 0}^n$
and $b_1,\dots,b_k \in \OO_A$
satisfying conditions (a) and (b) above with $m_{k+1}$ replaced
by $\lambda x_l$.
Put $y_{l,i} = \lambda^{-1}
b_i t^{I_i}$ and $x_{l+1} = x_l - y_{l,1} m_1 - \cdots
 - y_{l,k} m_k$.

If $|x_{l+1}|_1 = |x_l|_1$, we must have
$\deg(x_{l+1}) \prec \deg(x_l)$. Since $\preceq$ is a well ordering,
we must have $|x_{l'}|_1 < |x_l|_1$ for some $l' > l$.
Since $K$ is discretely valued and $A$ has the same group of nonzero norms as $K$,
we conclude that $|x_l|_1 \to 0$ as $l \to \infty$.

Since $|y_{l,i}|_1 \leq |x_l|_1$, we may set
$y_i = \sum_{l=0}^\infty y_{l,i}$ to get elements
of $T_n \widehat{\otimes}_{\Qp} A$
such that $x = y_1 m_1 + \cdots + y_k m_k$. This proves that
$J$ is always finitely generated,
so $T_n \widehat{\otimes}_{\Qp} A$ is noetherian.
\end{proof}

We next establish an analogue of the Nullstellensatz by combining the
previous argument
with an idea of Munshi \cite{M03}.
\begin{lemma}\label{lem:Nullstellensatz}
Take $K$ and $A$ as in Lemma~\ref{lem:noetherian},
but
suppose also that the intersection of the nonzero prime ideals of
$A$ is zero.
Then for any maximal ideal $\mathfrak{m}$ of
$T_n \widehat{\otimes}_{\Qp} A$,
the intersection $\mathfrak{m} \cap A$ is nonzero.
\end{lemma}
\begin{proof}
Suppose on the contrary that $\mathfrak{m}$ is a maximal
ideal of $T_n \widehat{\otimes}_{\Qp} A$
such that $\mathfrak{m} \cap A = 0$.
Since $T_n \widehat{\otimes}_{\Qp} A$ is noetherian by
Lemma~\ref{lem:noetherian}, $\mathfrak{m}$ is closed by
\cite[Proposition~3.7.2/2]{BGR84}.
Hence $\mathfrak{m} + A$ is also a closed subspace of
$T_n \widehat{\otimes}_{\Qp} A$.
Let $\psi: T_n \widehat{\otimes}_{\Qp} A \to (T_n \widehat{\otimes}_{\Qp} A)/A$
be the canonical projection; it
is a bounded surjective morphism of Banach spaces
with kernel $A$.
Put $V = \psi(\mathfrak{m} + A)$; since
$\mathfrak{m} + A = \psi^{-1}(V)$,
the open mapping theorem \cite[\S 2.8.1]{BGR84}
implies that $V$ is closed.
Hence $\psi$ induces a bounded bijective map
$\mathfrak{m} \to V$ between two Banach spaces;
by the open mapping theorem again, the inverse of $\psi$ is also bounded.

Using the power series representation of
elements of $T_n \widehat{\otimes}_{\Qp} A$, let us
represent $(T_n \widehat{\otimes}_{\Qp} A)/A$
as the set of series in $T_n \widehat{\otimes}_{\Qp} A$ with zero constant term.
We may then represent $\psi$
as the map that subtracts off the constant term.
Define the \emph{nonconstant degree}
of $x \in T_n \widehat{\otimes}_{\Qp} A$
as $\deg'(x) = \deg(\psi(x))$, and define the
\emph{leading nonconstant coefficient} of $x$ to be
the coefficient $x_{\deg'(x)}$.

We construct $m_1,\dots,m_k \in \mathfrak{m}$
using the following modified Buchberger algorithm.
As long as possible, choose an element $m_{k+1}$ of $\mathfrak{m} \cap \OO_A$
with nonconstant leading coefficient $a_{k+1}$,
for which we cannot choose
$I_1,\dots,I_k \in \ZZ_{\geq 0}^n$ and $b_1,\dots,b_k \in \OO_A$ satisfying
both of the following conditions.
\begin{enumerate}
\item[(a)]
We have $\deg'(m_{k+1}) = \deg'(m_i t^{I_i})$ whenever $b_i \neq 0$.
\item[(b)]
We have $a_{k+1} - a_1 b_1 - \cdots - a_k b_k \in I_A$.
\end{enumerate}
Again, this algorithm must terminate.

By the hypothesis on $A$, we can choose a nonzero
prime ideal $\mathfrak{p}$ of $A$
not containing the product $a_1 \cdots a_k$.
By our earlier hypothesis that $\mathfrak{m} \cap A = 0$,
we have $\mathfrak{m} \cap \mathfrak{p} = 0$.
Hence $\mathfrak{m} + \mathfrak{p}
(T_n \widehat{\otimes}_{\Qp} A)$ is the unit ideal,
so we can find $x_0 \in \mathfrak{p}(T_n \widehat{\otimes}_{\Qp} A)$
such that $1 + x_0 \in \mathfrak{m}$.

We now perform a modified division algorithm.
Given $x_l \in \mathfrak{p}(T_n \widehat{\otimes}_{\Qp} A)$
such that $1 + x_l \in \mathfrak{m}$,
we cannot have $x_l \in A$.
We may thus choose $\lambda \in A^\times$ with $|\psi(\lambda x_l)|_1 = 1$.
By the construction of $m_1,\dots,m_k$, there must exist
$I_1,\dots,I_k \in \ZZ_{\geq 0}^n$
and $b_1,\dots,b_k \in A$
satisfying conditions (a) and (b) above with $m_{k+1}$ replaced
by $\lambda x_l$.
Put $y_{l,i} = \lambda^{-1} b_i t^{I_i}$
and $x_{l+1} = x_l - y_{l,1} m_1 - \cdots - y_{l,k} m_k$.

As in the proof of Lemma~\ref{lem:noetherian}, we see that
$|\psi(x_l)|_1 \to 0$ as $l \to \infty$.
Since $\psi$ has bounded inverse, we also conclude that
$|x_l|_1 \to 0$ as $l \to \infty$. However, since $\mathfrak{m}$
is closed, this yields the contradiction $1 \in \mathfrak{m}$.
We conclude that $\mathfrak{m} \cap A \neq 0$, as desired.
\end{proof}

This yields the following.
\begin{lemma}\label{lem:Nullstellensatz-2}
For any Tate algebra $T_n$ over $\Qp$, any rational $s>0$,
and any complete discretely valued field extension $K$ of $\Qp$,
$T_n \widehat{\otimes}_{\Qp} \calE^s_{K}$ is
noetherian
and each of its maximal ideals has
residue field which is finite over $K$.
In particular, every maximal
ideal of $T_n \widehat{\otimes}_{\Qp} \calE^s_{K}$ lies over a maximal
ideal of $T_n$.
\end{lemma}
\begin{proof}
The Banach norm on $\calE^s_K$ is the maximum of the $p$-adic norm and the norm
induced by $w_s$. By enlarging $K$, we may assume that the nonzero values of this
norm are all achieved by elements of $K$. In this case,
we check that $A = \calE^s_K$ satisfies the hypotheses
of Lemma~\ref{lem:Nullstellensatz}. First, the nonzero norms of elements of $A$
are all realized by units of the form $\lambda t^i$ with
$\lambda \in K^\times$ and $i \in \ZZ$. Second, the residue ring
$\OO_A/I_A$ is isomorphic to a Laurent polynomial ring over a field,
which is noetherian. Third,
for each nonzero element $x$ of $A$,
we can construct $y \in A$ whose Newton polygon has no slopes
in common with that of $x$; this implies that $x$ and $y$ generate the
unit ideal
(e.g., see \cite[\S 2.6]{K05b}), so any maximal ideal containing $y$
fails to contain $x$. Hence the intersection of the nonzero prime ideals
of $A$ is zero; moreover, the quotient of $A$ by any nonzero ideal is finite over $K$.
We may thus apply Lemma~\ref{lem:Nullstellensatz} to deduce the claim.
\end{proof}

By combining Proposition~\ref{prop:noetherian} with Lemma~\ref{lem:Nullstellensatz-2}, we deduce the following. (The second assertion
follows from the first because for a coherent module, local freeness can be checked at
each maximal ideal.)
\begin{proposition}\label{prop:Kiehl}
For any $s>0$ and any finite extension $K$ of $\Qp$,
for $A = \calE^s_{K}$, any coherent $\A$-module $\mathcal{V}$ on $M(S)$
is associated to a finite $(S \widehat{\otimes}_{\Qp} A)$-module $V$.
Moreover, $\mathcal{V}$ is locally free if and only if $V$ is.
\end{proposition}
Using this, we may extend Theorem~\ref{thm:BC} for affinoid algebras, to eliminate
the hypothesis requiring a free Galois-stable lattice.
We first handle the case where $V_S$ is itself free.
\begin{theorem}\label{thm:BC-generalization}
Let $S$ be an affinoid algebra over $\Qp$.
Let $V_S$ be a free $S$-linear representation.
There exists
$s(V_S) \geq 0$ such that for $s\geq s(V_S)$,
we may construct a $(S\widehat{\otimes}_{\Qp} \mathbf{B}_K^{\dagger,s})$-module
$\D^{\dagger,s}_K(V_S)$ satisfying the following conditions.
\begin{enumerate}
\item[(1)]
The $(S\widehat{\otimes}_{\Qp} \mathbf{B}_K^{\dagger,s})$-module
$\D^{\dagger,s}_K(V_S)$ is locally free of rank $d$.
\item[(2)]The natural map $\D^{\dagger,s}_K(V_S)\otimes_{S\widehat{\otimes}_{\Qp} \mathbf{B}_K^{\dagger,s}} (S\widehat{\otimes}_{\Qp} \widetilde{\mathbf{B}}^{\dagger,s}) \ra V_S\otimes_S (S\widehat{\otimes}_{\Qp} \widetilde{\mathbf{B}}^{\dagger,s})$ is an isomorphism.
\item[(3)]For any maximal ideal $\mathfrak{m}_x$ of $S$,
for $V_x = V_S \otimes_S (S/\mathfrak{m}_x)$,
the natural map $\D^{\dagger,s}_K(V_S) \otimes_S (S/\mathfrak{m}_x)
\ra\D^{\dagger,s}_K(V_x)$ is an isomorphism.
\item[(4)] The construction is functorial in $V_S$, compatible with passage
from $K$ to a finite extension, and compatible with Theorem~\ref{thm:BC} in case $V_S$
admits a Galois-stable free lattice.
\end{enumerate}
\end{theorem}
\begin{proof}
Let $T_S$ be any free $\OO_S$-lattice in $V_S$.
Since the Galois action is continuous, there exists a finite Galois extension
$L$ of $K$ such that $G_L$ carries $T_S$ into itself.
For such $L$, for $s$ sufficiently large,
$\D^{\dagger,s}_L(V_S)$ is locally free of rank $d$
by Theorem~\ref{thm:BC}; moreover, it carries an action
of $\Gal(L/K)$.
If we restrict scalars on this module back to
$S\widehat{\otimes}_{\Qp} \mathbf{B}_K^{\dagger,s}$,
then $\D^{\dagger,s}_K(V_S)$ appears as a direct summand;
this summand is then finite projective, hence locally free
(since $T_n \widehat{\otimes}_{\Qp} \calE^s_{K}$ is
noetherian by Lemma~\ref{lem:Nullstellensatz-2}).
Moreover, the $\Gal(L/K)$-action on $\D^{\dagger,s}_L(V_S)$ allows
us to extend the $\Gamma_L$-action on $\D^{\dagger,s}_K(V_S)$ to
a $\Gamma_K$-action. This yields the desired assertions.
\end{proof}

\begin{definition}\label{def:fm-fg}
Let $S$ be an affinoid algebra over $\Qp$.
Let $V_S$ be a locally free $S$-linear representation;
we can then choose a finite covering of $M(S)$
by affinoid subdomains $M(S_1),\dots,M(S_n)$ such that
$V_i = V_S \otimes_S S_i$ is free over $S_i$ for each $i$.
We may then apply Theorem~\ref{thm:BC-generalization} to $V_i$ to produce
$\D^{\dagger,s}_K(V_i)$ for $s$ sufficiently large.
By Proposition~\ref{prop:Kiehl}, these glue to form a finite
$(S\widehat{\otimes}_{\Qp} \mathbf{B}_K^{\dagger,s})$-module
$\D^{\dagger,s}_K(V_S)$, which satisfies the analogues of the
assertions of Theorem~\ref{thm:BC-generalization}.
We may then define
\[
\D_K^\dagger(V_S)=\D^{\dagger,s}(V_S)\otimes_{S \widehat{\otimes}_{\Qp} \bdag{,s}_K} (S\widehat{\otimes}_{\Qp} \bb^\dagger_K),
\]
and this will be an \'etale $\m$-module over
$S\widehat{\otimes}_{\Qp} \bb^\dagger_K$.
The analogue of Proposition~\ref{prop:phi=1} will also carry over.
\end{definition}

\begin{remark}
Ga\"etan Chenevier has pointed out that Theorem~\ref{thm:BC-generalization} is also an easy
consequence of \cite[Lemme~3.18]{C09}.
That lemma
implies that for $S$ an affinoid algebra over $\Qp$
and $V_S$ a locally free $S$-linear representation,
there exist an affine formal scheme $\mathrm{Spf}(R)$ of finite type over $\Z$
equipped with an isomorphism $R \otimes_{\Z} \Qp \cong S$, and a
locally free $R$-linear representation $T_R$ admitting an isomorphism
$T_R \otimes_{\Zp} \Qp \cong V_S$.
This makes it possible to glue the Berger-Colmez theorem by doing so
on a suitable formal model of $S$.
\end{remark}

\section{Local coefficient algebras}

We next show that in a restricted setting, it is possible to invert the
$\m$-module functor $D^\dagger_K$.
\begin{definition}
By a \emph{coefficient algebra}, we mean
a commutative Banach algebra $S$ over $\Q$
satisfying the following conditions.
\begin{itemize}
\item[(i)]
The norm on $S$ restricts to the norm on $\Q$.
\item[(ii)]
For each maximal ideal $\mathfrak{m}$ of $S$, the residue field of $\mathfrak{m}$
is finite over $\Q$.
\item[(iii)]
The Jacobson radical of $S$ is zero; in particular, $S$ is reduced.
\end{itemize}
For instance, any reduced affinoid algebra over $\Q$ is a coefficient
algebra.

By a \emph{local coefficient algebra}, we will mean a coefficient algebra $S$
of the form $R \otimes_{\Zp} \Qp$, where $R$ is a complete noetherian
local domain
of characteristic $0$ with residue field finite over $\mathbb{F}_p$.
For instance, if $S$ is a reduced affinoid algebra over $\Qp$
equipped with the spectral
norm, and $R$ is the completion of $\OO_S$ at a maximal ideal, then
$R \otimes_{\Zp} \Qp$ is a local coefficient algebra.
\end{definition}

One special property of local coefficient algebras is the following.
(Compare the discussion preceding Lemma~\ref{lem:phi=1}.)
\begin{proposition}\label{prop:c-algebra}
Let $R$ be a complete noetherian local domain
of characteristic $0$ with residue field finite over $\mathbb{F}_p$,
and let $S$ be the local coefficient algebra $R \otimes_{\Zp} \Qp$.
\begin{enumerate}
\item[(a)]
We may naturally identify
$(R \widehat{\otimes}_{\Zp} \aa_K) \otimes_{\Zp} \Qp$
with the $p$-adic completion
of $S \widehat{\otimes}_{\Qp} \bb^\dagger_K$.
\item[(b)]
We may naturally identify
$(R \widehat{\otimes}_{\Zp} \at_K) \otimes_{\Zp} \Qp$
with a subring of the $p$-adic completion of
$S \widehat{\otimes}_{\Qp} \bt^\dagger_K$.
\end{enumerate}
\end{proposition}
\begin{proof}
Let $P_{1,n,s}, P_{2,n,s}, P_{3,n,s}$ denote the completed tensor products
$(R/p^n R) \widehat{\otimes}_{\Zp} (\aa^{\dagger,s}/p^n \aa^{\dagger,s})$
formed using the following choices for the topologies on the two sides.
\begin{itemize}
\item
For $P_{1,n,s}$, use on the left side the discrete topology, and on the right side the
topology induced by $w_s$.
\item
For $P_{2,n,s}$, use on the left side the topology induced by $\mathfrak{m}_R$, and on the right side the
topology induced by $w_s$.
\item
For $P_{3,n,s}$, use on the left side the topology induced by $\mathfrak{m}_R$, and on the right side the
discrete topology.
\end{itemize}
These constructions relate to our original question as follows. If we take the inverse limit of
the $P_{1,n,s}$ as $n  \to \infty$, then invert $p$,
then take the union over all choices of $s$, we recover $S \widehat{\otimes}_{\Qp} \bb^\dagger$.
If we take the union of the $P_{3,n,s}$ over all choices of $s$, then take the inverse limit as $n \to \infty$,
and finally invert $p$, we recover $(R \widehat{\otimes}_{\Zp} \aa) \otimes_{\Zp} \Qp$.

To establish (a), it thus suffices to check that the natural maps $P_{1,n,s} \to P_{2,n,s}$ and $P_{3,n,s} \to P_{2,n,s}$ are both
bijections. Put $A = \mathfrak{m}_R (R/p^n R)$ and $I = \mathfrak{m}_R A$. Put
$B = \aa^{\dagger,s}/p^n \aa^{\dagger,s}$ and choose an ideal of definition $J \subseteq B$
for the topology induced
by $w_s$. In this notation, $A$ is $I$-adically complete and separated, $B$ is
$J$-adically complete and separated, and both $A$ and $B$ are flat over
$\Z/p^n \Z$. Put $C = A \otimes_{\Z/p^n\Z} B$.
The $IC$-adic completion of $C$ is then the inverse limit over $m$ of the quotients
$C/I^m C = (A/I^m A) \otimes_{\Z/p^n\Z} B$.
Since $B$ is flat over $\Z/p^n \Z$ and $A/I^m A$ has finite cardinality,
the completeness of $B$ with respect to $J$ implies the completeness of
$C/I^m C$ with respect to $J(C/I^m C)$. It follows that $C$ is complete with respect
to $IC + JC$, which means that $P_{1,n,s} \to P_{2,n,s}$ is a bijection.
Similarly, we may argue that $P_{3,n,s} \to P_{2,n,s}$ is bijective using the fact that
$B/J^m B$ is of finite cardinality.

This yields (a). The whole argument carries over in the case of (b) except for the finiteness
of $B/J^m B$; hence in this case, we only have that $P_{1,n,s} \to P_{2,n,s}$ is a bijection
and $P_{3,n,s} \to P_{2,n,s}$ is injective.
\end{proof}

\begin{theorem}\label{thm:fg-convert-coefficient}
Let $S$ be a local coefficient algebra.
Let $M_S$ be an \'etale $\m$-module over
$S\widehat{\otimes}_{\Qp} \bb^\dagger_{K}$,
and put
\[
V_S=(M_S \otimes_{S\widehat{\otimes}_{\Qp} \bb^\dagger_{K}}
(S\widehat{\otimes}_{\Qp} \bt^\dagger))^{\varphi=1}.
\]
Then $V_S$ is an $S$-linear representation for which
the natural map $\D^\dagger_{K}(V_S) \to M_S$ is an isomorphism.
\end{theorem}
\begin{proof}
By Proposition~\ref{prop:c-algebra}(a), we may identify the $p$-adic completion of
$S\widehat{\otimes}_{\Qp} \bb^\dagger_{K}$
with $(R \widehat{\otimes}_{\Zp} \aa) \otimes_{\Zp} \Qp$. This allows us to define
\[
V'_S = (M_S \otimes_{S\widehat{\otimes}_{\Qp} \bb^\dagger_{K}}
((R \widehat{\otimes}_{\Zp} \at) \otimes_{\Zp} \Qp))^{\varphi=1}.
\]
By Theorem~\ref{thm:Dee}, the natural map
\[
V'_S \otimes_{(R \widehat{\otimes}_{\Zp} \aa) \otimes_{\Zp} \Qp}
((R \widehat{\otimes}_{\Zp} \at) \otimes_{\Zp} \Qp)
\to M_S \otimes_{S\widehat{\otimes}_{\Qp} \bb^\dagger_{K}}
((R \widehat{\otimes}_{\Zp} \at) \otimes_{\Zp} \Qp)
\]
is an isomorphism.

By Proposition~\ref{prop:c-algebra}(b), we may identify
$(R \widehat{\otimes}_{\Zp} \at) \otimes_{\Zp} \Qp$ with a subring of
the $p$-adic completion of $S\widehat{\otimes}_{\Qp} \bt^\dagger$.
Using this identify, we may argue as in
\cite[Proposition~1.2.7]{K08} to show that $V'_S \subseteq V_S$,
which is enough to establish the desired result.
\end{proof}

\section{A lifting argument}

While one cannot invert the functor $\D^\dagger_K$ for an arbitrary $S$,
one can give a partial result.
\begin{lemma}\label{lem:dejong}
For any commutative Banach algebra $S$ over $\Qp$, any $s>0$, and any
$x \in S \widehat{\otimes}_{\Qp} \bt^{\dagger,s}$, the equation
\[
y - \varphi^{-1}(y) = x
\]
has a solution $y \in S \widehat{\otimes}_{\Qp} \bt^{\dagger,s}$.
More precisely,
we may choose $y$ such that $v_p(y) \geq v_p(x)$
and $w_s(y) \geq w_s(x)$.
\end{lemma}
\begin{proof}
For $S = \Qp$, the existence of a solution $y \in \bt$ follows from
the fact that $\bt$ is a complete discretely valued field with
algebraically closed residue field.
Write $x = \sum_k p^k [x_k]$ and $y = \sum_k p^k [y_k]$.
We claim that $y$ can be chosen such that for each $k$,
\[
\inf\{v_{\et}(y_\ell): \ell \leq k\}
\geq
\inf\{v_{\et}(x_\ell): \ell \leq k\},
\]
which yields the desired results.
This choice can be made
because for any $\overline{x} \in \et$, the equation
$\overline{y} - \overline{y}^{1/p} = \overline{x}$ always has a solution
$\overline{y} \in \et$ with
\[
v_{\et}(\overline{y}) \geq \begin{cases}
v_{\et}(\overline{x}) & v_{\et}(\overline{x}) \leq 0 \\
p v_{\et}(\overline{x}) & v_{\et}(\overline{x}) \geq 0.
\end{cases}
\]

For general $S$, write $x$ as a convergent sum $\sum_i u_i \otimes x_i$
with $u_i \in S$ and $x_i \in \bt^{\dagger,s}$. For each $i$, let
$y_i \in \bt^{\dagger,s}$ be a solution of $y_i - \varphi^{-1}(y_i) = x_i$
with $w_s(y_i) \geq w_s(x_i)$. Then
the sum $y = \sum_i u_i \otimes y_i$ converges
with the desired effect.
\end{proof}

\begin{theorem}\label{thm:fg-convert}
Let $S$ be a commutative Banach algebra over $\Qp$.
Let $M_S$ be a free \'etale $\m$-module over
$S\widehat{\otimes}_{\Qp} \bb^\dagger_{K}$.
Suppose that there exists a basis of $M_S$ on
which $\varphi-1$ acts via a matrix whose entries have positive $p$-adic
valuation.
Then
\[
V_S=(M_S \otimes_{S\widehat{\otimes}_{\Qp} \bb^\dagger_{K}}
(S\widehat{\otimes}_{\Qp} \bt^\dagger))^{\varphi=1}
\]
is a free $S$-linear representation for which
the natural map $\D^\dagger_{K}(V_S) \to M_S$ is an isomorphism.
\end{theorem}
\begin{proof}
Choose a basis of
$M'_S = M_S \otimes_{S\widehat{\otimes}_{\Qp} \bb^\dagger_{K}}
(S\widehat{\otimes}_{\Qp} \bt^\dagger)$
on which $\varphi-1$ acts via a matrix $A$ whose
entries belong to $S \widehat{\otimes}_{\Qp} \bb^{\dagger,s}_K$ for some $s>0$
and have $p$-adic valuation bounded below by $c>0$.
We may apply Lemma~\ref{lem:dejong} to choose a matrix $X$ such that
$X$ has entries in $S \widehat{\otimes}_{\Qp} \bb^{\dagger,s}_K$
with $p$-adic valuation bounded below by $c$,
$\min_{i,j} \{w_s(X_{i,j})\} \geq \min_{i,j} \{w_s(A_{i,j})\}$,
and $X - \varphi^{-1}(X) = A$. We can thus change basis to get a
new basis of $M'_S$ on which $\varphi-1$ acts via the matrix
\[
(I_n - \varphi^{-1}(X))^{-1}(I_n + A)(I_n - X) - I_n,
\]
whose entries have valuation bounded below by $2c$.
If we repeat this process, we get a sequence of matrices $X_1, X_2, \dots$
such that $w_s(X_i)$ is bounded below, and the $p$-adic valuation of $X_i$
is at least $ci$. It follows that $w_{s'}(X_i)$ tends to infinity
for any $s' > s$, so the product
$(I_n + X_1)(I_n + X_2) \cdots$ converges in
$S \widehat{\otimes}_{\Qp} \bb^{\dagger,s'}_K$ and defines a basis of $M'_S$ fixed by
$\varphi$.
This proves the claim.
\end{proof}

\begin{remark}
The hypothesis about the basis of $M_S$ is needed in Theorem~\ref{thm:fg-convert}
for the following reason. For $R$ an arbitrary $\mathbb{F}_p$-algebra,
if $\varphi$ acts as the identity on $R$ and as the $p$-power Frobenius
on $\et$, given an invertible square matrix $A$ over
$R \otimes_{\mathbb{F}_p} \et$,
we cannot necessarily solve the matrix equation $U^{-1} A \varphi(U) = A$
for an invertible matrix $U$ over $R \otimes_{\mathbb{F}_p} \et$.
For instance, in Chenevier's example, there is no
solution of the equation $\varphi(z) = Yz$.

One may wish to view the collection of isomorphism classes
of $\m$-modules over $R \otimes_{\mathbb{F}_p} \mathbb{F}_p((\epsilon - 1))$,
for $R$ an $\mathbb{F}_p$-algebra, as the ``$R$-valued points of the
moduli space of
mod $p$ representations of $G_{\Qp}$''. To replace $\Qp$ with $K$,
one should replace $\mathbb{F}_p((\epsilon - 1))$
with the $H_K$-invariants of its separable closure.
\end{remark}

\section{Families of $\m$-modules and \'etale models}

We would like to turn next from  $\m$-modules over
$S \widehat{\otimes}_{\Qp} \bb^\dagger_{K}$
to $\m$-modules over
$S \widehat{\otimes}_{\Qp} \bb^\dagger_{\rig,K}$.
In the absolute case, these have important applications
to the study of de Rham representations, as shown by Berger;
see for instance \cite{B04}.
In the relative case, however, they do not form a robust enough category
to be useful; it is better to pass to a more geometric notion.
For this, we must restrict to the case
where $S$ is an affinoid algebra.

\begin{definition}
Let $K$ be a finite extension of $\Qp$,
and let $S$ be an affinoid algebra over $K$.
Recall that $\calR^s_{K}$ denotes the ring
of Laurent series with coefficients in $K$ in a variable $T$ convergent
on the annulus $0 < v_p(T) \leq 1/s$, and that $\calR_{K}=\cup_{s>0}\calR^s_{K}$.
By a \emph{vector bundle} over $S \widehat{\otimes}_{K} \calR^s_{K}$,
we will mean a coherent locally free sheaf over the product
of this annulus with $M(S \otimes_{K} K)$
in the category of rigid analytic spaces over $K$.
(In case $S$ is disconnected, we insist that the rank be constant,
not just locally constant.)
By a vector bundle over $S \widehat{\otimes}_{K}
\calR_{K}$, we will mean an object in the direct limit
as $s \to \infty$ of the categories of vector bundles over
$S \widehat{\otimes}_{\Qp} \calR^s_{K}$.

Recall that for $s$ sufficiently large,
we can produce an isomorphism $\bb^{\dagger,s}_{\rig,K} \cong
\calR^s_{K'_0}$. We thus obtain the notion of a vector bundle
over $S \widehat{\otimes}_{\Qp} \bb^{\dagger,s}_{\rig,K}$,
dependent on the choice of the isomorphism. However,
the notion of a vector bundle
over $S \widehat{\otimes}_{\Qp} \bb^{\dagger}_{\rig,K}$
does not depend on any choices.
\end{definition}

\begin{remark}\label{rem:freeness}
For $S = K$ discretely valued, every vector bundle over $\calR^s_{K}$ is freely
generated by global sections
\cite[Theorem~3.4.1]{K05a}.
On the other hand, for $S$ an affinoid algebra over $\Qp$,
we do not know whether
any vector bundle over $S \widehat{\otimes}_{\Qp} \r_K^s$ is $S$-locally free;
this does not follow from the work of L\"utkebohmert
\cite{L77}, which only applies to closed annuli.
\end{remark}

\begin{definition}
Let $K$ be a finite extension of $\Qp$,
and let $S$ be an affinoid algebra over $\Qp$.
By a \emph{family of $\m$-modules} over $S \widehat{\otimes}_{\Qp}
\bb^{\dagger}_{\rig,K}$, we will mean a vector bundle $V$ over
$S \widehat{\otimes}_{\Qp} \bb^{\dagger}_{\rig,K}$
equipped with an isomorphism $\varphi^* V \to V$,
viewed as a semilinear $\varphi$-action,
and a semilinear $\Gamma$-action commuting with the $\varphi$-action.
We say a family of $\m$-modules over $S \widehat{\otimes}_{\Qp}
\bb^{\dagger}_{\rig,K}$ is \emph{\'etale} if it arises by base extension
from an \'etale $\m$-module over $S \widehat{\otimes}_{\Qp}
\bb^{\dagger}_{K}$; we call the latter an \emph{\'etale model} of the family.
\end{definition}

It turns out that \'etale models are unique when they exist.
To check this without any reducedness hypothesis on $S$,
we need a generalization of the fact that a reduced affinoid algebra
embeds into a product of complete fields \cite[Proposition~2.4.4]{B90}.
\begin{lemma}\label{lem:affinoid-inclusion}
Let $K$ be a finite extension of $\Qp$,
and let $S$ be an affinoid algebra over $\Qp$.
Then there exists a strict inclusion
$S \to \prod_{i=1}^n A_i$ of topological rings, in which each
$A_i$ is a finite connected algebra over a complete discretely valued field.
\end{lemma}
\begin{proof}
Let $T$ be the multiplicative subset of $\OO_S$ consisting of
elements whose images in $\OO_S/I_S$ are not zero divisors.
For any $s \in S$ and $t \in T$, we have $|st| = |s| |t|$, so the norm
on $S$ extends uniquely to the localization $S[T^{-1}]$.
The completion of this localization has the desired form.
\end{proof}
\begin{proposition}\label{prop:faithful}
Let $K$ be a finite extension of $\Qp$,
and let $S$ be an affinoid algebra over $\Qp$.
Then the natural base change functor from \'etale $\m$-modules over
$S \widehat{\otimes}_{\Qp} \bb^{\dagger}_{K}$
to families of $\m$-modules over
$S \widehat{\otimes}_{\Qp} \bb^{\dagger}_{\rig,K}$
is fully faithful. In fact, this holds even without the $\Gamma$-action.
\end{proposition}
\begin{proof}
Note that if we replace $S$ by a complete discretely valued field $L$,
we may deduce the analogous claim by \cite[Theorem~6.3.3]{K05b}
after translating notations.
(We must note that families of $\m$-modules over
$\bb^{\dagger}_{\rig,K}$
are finite free over $\bb^{\dagger}_{\rig,K}$
by Remark~\ref{rem:freeness}.)
In fact, if we replace $S$ by a finite algebra over $L$, we may
make the same deduction by restricting scalars to $L$.
We may thus deduce the original claim by embedding $S$ into a product
of finite algebras over complete discretely valued fields
using Lemma~\ref{lem:affinoid-inclusion}.
\end{proof}
\begin{corollary}\label{cor:em-unique}
Let $K$ be a finite extension of $\Qp$,
and let $S$ be an affinoid algebra over $\Qp$.
Then an \'etale model of a family of $\m$-modules over
$S \widehat{\otimes}_{\Qp} \bb^{\dagger}_{\rig,K}$
is unique if it exists.
\end{corollary}

\begin{definition}
Let $S$ be an affinoid algebra over $\Qp$.
Let $V_S$ be a locally free $S$-linear representation.
We define $\D^\dagger_K(V_S)$ as in Definition~\ref{def:fm-fg}, then put
\[
\D^\dagger_{\rig,K}(V_S) = \D^\dagger_K(V_S) \otimes_{S
\widehat{\otimes}_{\Qp} \bb^{\dagger}_{K}} (S
\widehat{\otimes}_{\Qp} \bb^{\dagger}_{\rig, K}).
\]
This is an \'etale $\m$-module over $S
\widehat{\otimes}_{\Qp} \bb^{\dagger}_{\rig, K}$,
from which we may recover $V_S$ by taking
\[
V_S=(\D^\dagger_{\rig,K}(V_S) \otimes_{S\widehat{\otimes}_{\Qp} \bb^\dagger_{\rig,K}}
(S\widehat{\otimes}_{\Qp} \bt^\dagger_{\rig}))^{\varphi=1}.
\]
\end{definition}
We may now obtain Theorem~\ref{thm:A} of the introduction by
combining Theorem~\ref{thm:fg-convert-coefficient} (via Definition~\ref{def:fm-fg}) with Proposition~\ref{prop:faithful}.

\section{Local \'etaleness}

We now turn to Theorem~\ref{thm:B} of the introduction. Given what we already
have proven, this can be obtained by invoking some results
from \cite{R08}. For the convenience of the reader, we recall
these results in detail.

\begin{lemma}\label{lem:f-ngbd}
Let $K$ be a finite extension of $\Qp$,
and let $S$ be an affinoid algebra over $K$.
For any $x\in M(S)$ and $\lambda>0$, there exists an affinoid
subdomain $M(B)$ of $M(S)$ containing $x$ such that if $f\in S$ vanishes at
$x$, then $|f(y)|\leq\lambda|f|_S$ for any $y\in M(B)$.
\end{lemma}
\begin{proof}
We first prove the lemma for $S=T_n=K\langle x_1,\dots,x_n\rangle$, the
$n$-dimensional Tate algebra over $K$. It is harmless to enlarge $K$,
so we may suppose without loss of generality
that $x$ is the origin $x_1=\cdots=x_n=0$. Choosing a rational
number $\lambda'<\lambda$, the affinoid subdomain
$\{(x_1,\dots,x_n) \in M(S): |x_1|\leq\lambda', \dots, |x_n|\leq\lambda'\}$
satisfies the required property.

For general $S$, the reduction $\overline{S}=\OO_S/\mathfrak{m}_K\OO_S$ is a
finite type scheme over the residue field $k$ of $K$.
For $n$ sufficiently large, we take a
surjective $k$-algebra homomorphism
$\overline{\alpha}:k[\overline{x_1},\dots,\overline{x_n}]\twoheadrightarrow
\overline{S}$. We lift
$\overline{\alpha}$ to a $K$-affinoid algebra homomorphism
$\alpha:K\langle x_1,\dots,x_n\rangle\ra S$ by mapping $x_i$ to a
lift of $\overline{\alpha}(\overline{x_i})$ in $\OO_S$. Then it
follows from Nakayama's lemma that $\alpha$ maps $\OO_K\langle
x_1,\dots,x_n\rangle$ onto $\OO_S$. Let $\alpha$ also denote
the induced map from $M(S)$ to $M(K\langle
x_1,\dots,x_n\rangle)$. By the case of $K\langle
x_1,\dots,x_n\rangle$, we can find an affinoid neighborhood $M(B)$
of $\alpha(x)$ satisfying the required property for $\lambda/p$. Now for any nonzero $f\in
S$ vanishing at $x$, we choose $c\in \Q$ such that $|c|\leq|f|_S\leq p|c|$, yielding $pf/c\in\OO_S$. Pick $f'\in\OO_K\langle
x_1,\dots,x_n\rangle$ such that $\alpha(f')=pf/c$. Then
$f'(\alpha(x))=(pf/c)(x)=0$ implies that $|f'(y)|\leq(\lambda/p)|f'|_{T_n}\leq \lambda/p$ for
any $y\in M(B)$. Then for any $y\in \alpha^{-1}(M(B))$, we have
$|pf(y)|/|c|=|f'(\alpha(y))|\leq\lambda/p$, yielding $|f(y)|\leq\lambda|c|\leq\lambda|f|_S$. Hence $\alpha^{-1}(M(B))$
is an affinoid neighborhood of $x$ satisfying the property we need.
\end{proof}

\begin{definition}
For $S$ a commutative Banach algebra over $\Qp$ and
$I$ a subinterval of $\mathbb{R}$,
let $\r_S^I$ be the ring of Laurent series
over $S$ in the variable $T$ convergent for $v(T)^{-1} \in I$.
Let $v_S$ be the valuation on $S$, and for $s \in I$ and $x = \sum_i
x_i T^i \in \r_S^I$ put
\[
w_s(x) = \inf_i \{ i + s v_S(x_i)\}.
\]
Put $\r_S^s = \r_S^{[s, +\infty)}$,
which we may identify with
the completed tensor product $S \widehat{\otimes}_{\Qp} \r_{\Qp}^s$
for the Fr\'echet topology on the right, and put $\r_S=\cup_{s>0}\r_S^s$.
Let $\r_S^{\int,s}$ be the subring of $\r_S^s$ consisting of series
with coefficients in $\OO_S$.
\end{definition}

The following lemma is based on \cite[Lemma~6.1.1]{K05b}.
\begin{lemma}\label{lem:rig-to-dagger}
Let $K$ be a finite extension of $\Qp$,
and let $S$ be an affinoid algebra over $K$.
Pick $s_0 > 0$.
Let $\varphi: \calR^{s_0/p}_S \to \calR^{s_0}_S$
be a map of the form $\sum_i c_i T^i \mapsto \sum_i \phi_S(c_i) W^i$,
where $\phi_S: S \to S$ is an isometry and $W \in
\calR^{s_0}_S$ satisfies
$w_{s_0}(W -T^p) > w_{s_0}(T^p)$.
For some $s \geq s_0$,
let $D$ be an
invertible $n\times n$ matrix over $\r_S^{[s,s]}$, and put
$h=-w_s(D)-w_s(D^{-1})$; it is clear that $h\geq0$. Let $F$ be an $n\times n$ matrix over
$\r_S^{[s,s]}$ such that $w_s(FD^{-1}-I_n)\geq c+h/(p-1)$ for a
positive number $c$. Then for any positive integer $k$ satisfying
$2(p-1)sk\leq c$, there exists an invertible $n\times n$ matrix
$U$ over $\r_S^{[s/p,s]}$ such that
$U^{-1}F\varphi(U)D^{-1}-I_n$ has
entries in $p^k\r_S^{\int, s}$ and
$w_s(U^{-1}F\varphi(U)D^{-1}-I_n)\geq c+h/(p-1)$.
\end{lemma}

\begin{proof}
For
$i\in\mathbb{R}$, $s>0$, $f=\sum^{+\infty}_{j=-\infty}a_j
T^j\in\r_S$, we set $v_i(f)=\min\{j:v_S(a_j)\leq i\}$ and
$v_{i,s}(f)=v_i(f)+si$. It is clear that $v_{i,s}(f)\geq w_s(f)$. (In case $S$ is a field, these
are similar to the quantities
$v_i^{\rm{naive}}$, $v_{i,r}^{\rm{naive}}$ in \cite[p. 458]{K05b},
albeit with a slightly different normalization.)

We define a sequence of invertible matrices $U_0$, $U_1$, $\dots$ over
$\r_S^{[s/p,s]}$ and a sequence of matrices $F_0$, $F_1$, $\dots$ over
$\r_S^{[s,s]}$ as follows. Set $U_0=I_n$. Given $U_l$, put
$F_l=U_l^{-1}F\varphi(U_l)$. Suppose
$F_lD^{-1}-I_n=\displaystyle{\sum_{m=-\infty}^{\infty}}V_mT^m$ where
the $V_m$'s are $n\times n$ matrices over $S$. Let
$X_l=\displaystyle{\sum_{v_S(V_m)< k}}V_mT^m$, and put
$U_{l+1}=U_l(I_n+X_l)$. Set
\[
c_l=\inf_{i\leq k-1}\{{v_{i,s}(F_lD^{-1}-I_n)}-h/(p-1)\}.
\]
We now prove by induction that $c_l\geq\frac{l+1}{2}c$,
$w_{s}(F_lD^{-1}-I_n)\geq c+h/(p-1)$ and $U_l$ is invertible over
$\r_S^{[s/p,s]}$ for any $l\geq0$. This is obvious for $l=0$. Suppose
that the claim is true for some $l\geq 0$.  Then for any $t\in[s/p,s]$,
since $c_l\geq\frac{l+1}{2}c\geq (p-1)sk$, we have
\begin{center}
$w_t(X_l)\geq w_s(X_l)-(s-t)k\geq
(c_l+h/(p-1))-(s-t)k>0$.
\end{center}
Hence $U_{l+1}$ is also invertible over $\r_S^{[s/p,s]}$.
Furthermore, we have
\begin{eqnarray*}
w_s(D\varphi(X_l)D^{-1})&\geq& w_s(D)+w_s(\varphi(X_l))+w_s(D^{-1})\\
&=&p w_{s/p}(X_l)-h\\
&>&p(c_l+h/(p-1))-h-(p-1)sk\\
&=&pc_l+h/(p-1)-(p-1)sk\\
&\geq& c_l+\frac{1}{2}c+h/(p-1)+(\frac{1}{2}c-(p-1)sk)\\
&\geq& \frac{(l+2)}{2}c+h/(p-1).
\end{eqnarray*}
Since $c_l\geq\frac{l+1}{2}c$ by inductive assumption. Note that
\begin{eqnarray*}
F_{l+1}D^{-1}-I_n&=&(I_n+X_l)^{-1}F_lD^{-1}(I_n+D\varphi(X_l) D^{-1})-I_n\\
&=&((I_n+X_l)^{-1}F_lD^{-1}-I_n)+(I_n+X_l)^{-1}(F_lD^{-1})D\varphi(X_l)D^{-1}.
\end{eqnarray*}
Since $w_s(F_lD^{-1})\geq0$ and $w_s((I_n+X_l)^{-1})\geq0$, we have $w_s((I_n+X_l)^{-1}(F_lD^{-1})D\varphi(X_l)D^{-1})\geq \frac{(l+2)}{2}c+h/(p-1)$. Write
\begin{eqnarray*}
(I_n+X_l)^{-1}F_lD^{-1}-I_n&=&(I_n+X_l)^{-1}(F_lD^{-1}-I_n-X_l)\\
&=&\displaystyle{\sum_{j=0}^{\infty}}(-X_l)^{j}(F_lD^{-1}-I_n-X_l).
\end{eqnarray*}
For $j\geq1$, we have
\begin{center}
$w_s((-X_l)^{j}(F_lD^{-1}-I_n-X_l))\geq c+c_l+2h/(p-1)>
\frac{l+2}{2}c+h/(p-1)$.
\end{center}
By the definition of $X_l$, we also have $v_i(F_lD^{-1}-I_n-X_l)=\infty$
for $i< k$ and $w_{s}(F_lD^{-1}-I_n-X_l)\geq c+h/(p-1)$.
Putting these together, we get that
\[
v_{i,s}(F_{l+1}D^{-1}-I_n)\geq \frac{l+2}{2}c+h/(p-1)
\]
for any $i< k$, i.e., $c_{l+1}\geq \frac{l+2}{2}c$, and that
$w_{s}(F_{l+1}D^{-1}-I_n)\geq c+h/(p-1)$. The induction step is
finished.

Now since $w_t(X_l)\geq c_l+h/(p-1) - (p-1)ps/k$ for $t\in[s/p,s]$, and
$c_l\rightarrow\infty$ as $l\rightarrow\infty$, the sequence ${U_l}$
converges to a limit $U$, which is an invertible $n\times n$ matrix
over $\r_S^{[s/p,s]}$ satisfying $w_s(U^{-1}F\varphi(U)D^{-1}-I_n)\geq
c+h/(p-1)$. Furthermore, we have
\begin{center}
$v_{m,s}(U^{-1}F\varphi(U)D^{-1}-I_n)=\displaystyle{\lim_{l\rightarrow\infty}}
v_{m,s}(U_l^{-1}F\varphi(U_l)D^{-1}-I_n)
=\displaystyle{\lim_{l\rightarrow\infty}}
v_{m,s}(F_{l+1}D^{-1}-I_n)=\infty$,
\end{center}
for any $m< k$. Therefore $U^{-1}F\varphi(U)D^{-1}-I_n$ has
entries in $p^k\r_S^{\int, s}$.
\end{proof}

\begin{theorem}\label{thm:rigngbd}
Let $S$ be an affinoid algebra over $\Qp$, and let $M_S$ be a family of $\m$-modules over $S \widehat{\otimes}_{\Qp}
\bb^\dagger_{\rig,K}$, such that for some $x \in M(S)$ whose residue field is contained in $S$,
the fibre
$M_x$ of $M_S$ over $x$ is \'etale. Then there
exists an affinoid neighborhood $M(B)$ of $x$ and a finite extension of
$L$ of $K$ such that the base extension $M_B$ of $M_S$ to $B
\widehat{\otimes}_{\mathbb{Q}_p} \mathbf{B}^{\dagger}_{\mathrm{rig},L}$
has an \'etale model in which the entries of the matrix of $\varphi-1$ have positive $p$-adic valuation.
\end{theorem}
\begin{proof}
Because Proposition~\ref{prop:faithful} does not require the $\Gamma$-action,
it suffices to construct an \'etale model just for the $\varphi$-action.
Choose an isomorphism $\bb^{\dagger,s_0}_{\rig,K} \cong \calR^{s_0}_{K'_0}$
for some $s_0>0$, via which $\varphi$ induces a map
from $\calR^{s_0/p}_{K'_0}$ to $\calR^{s_0}_{K'_0}$ satisfying
$w_{s_0}(\varphi(T) - T^p) > w_{s_0}(T^p)$.
Then choose $s\geq s_0$ such that $M_S$ is represented by a
vector bundle $V_S$ over $S \widehat{\otimes}_{\Qp} \calR^{s/p}_{K'_0}$
equipped with an isomorphism $\varphi^* V_S \to V_S$
of vector bundles over $S \widehat{\otimes}_{\Qp} \calR^{s}_{K'_0}$.

By hypothesis, $M_x$ is \'etale. After increasing $s$, we may thus
assume that $M_x$ admits a basis $e_x$ on which $\varphi$ acts
via an invertible matrix over $\calR_{S/\mathfrak{m}_x}^{\int,s}$.
Lift this matrix to a matrix $D$
over $\calR_S^{\int,s}$, using the inclusion $S/\mathfrak{m}_x \hookrightarrow S$
which was assumed to exist. By enlarging $K$, we can ensure that
$D-1$ has positive $p$-adic valuation (by first doing so modulo $\mathfrak{m}_x$).

By results of L\"utkebohmert \cite[Satz 1, 2]{L77},
the restriction of $V_S$ to $S \widehat{\otimes}_{\Qp} \calR^{[s/p,s]}_{K'_0}$
is $S$-locally free. By replacing $M(S)$ with an affinoid subdomain containing
$x$, we may reduce to the case where this restriction admits a basis
$e_S$. Let $A$ be the matrix via which $\varphi$ acts on this basis;
it has entries in $S \widehat{\otimes}_{\Qp} \calR^{[s,s]}_{K'_0}$.
Let $V$ be a matrix over
$S \widehat{\otimes}_{\Qp} \calR^{[s/p,s]}_{K'_0}$
lifting (again using the inclusion $S/\mathfrak{m}_x \hookrightarrow S$)
the change-of-basis matrix from the mod-$\mathfrak{m}_x$ reduction of $e_S$
to $e_x$.

By Lemma~\ref{lem:f-ngbd}, we can shrink $S$ so as to make $D$ invertible over
$\calR_S^{\int,s}$. We can also force $V$ to become invertible,
and we may make $V^{-1} A \varphi(V) - D$ as small as desired.
We may thus put ourselves in position to apply Lemma~\ref{lem:rig-to-dagger}
with $F = V^{-1} A \varphi(V)$, to produce an invertible
$n \times n$ matrix $U$ over $S \widehat{\otimes}_{\Qp} \calR^{[s/p,s]}_{K'_0}$
such that $W = U^{-1} F \varphi(U) D^{-1} - I_n$ has entries in
$p \OO_S \widehat{\otimes}_{\Zp} \calR^{\int,s}_{K'_0}$
and $w_s(W) > 0$.

Changing basis from $e_S$ via the matrix $VU$ gives another basis
$e'_S$ of $V_S$ over $S \widehat{\otimes}_{\Qp} \calR^{[s/p,s]}_{K'_0}$,
on which $\varphi$ acts via the matrix $(W + I_n)D$.
We may change basis $e_S'$ using $(W+I_n)D$ to get a new basis of
$V_S$ over $S \widehat{\otimes}_{\Qp} \calR^{[s,ps]}_{K'_0}$;
since $(W+I_n)D$ is invertible over
$\OO_S \widehat{\otimes}_{\Zp} \calR^{\int,s}_{K'_0}$,
the basis $e'_S$ also generates $V_S$ over
$S \widehat{\otimes}_{\Qp} \calR^{[s,ps]}_{K'_0}$.
Repeating the argument, we deduce that $e'_S$ is actually a basis of
$V_S$ generating an \'etale model. This proves the claim.
\end{proof}
Combining Theorem~\ref{thm:fg-convert} with Theorem~\ref{thm:rigngbd} yields Theorem~\ref{thm:B}. Note that before applying
Theorem~\ref{thm:rigngbd}, we must first extend scalars from $S$ to $S \otimes_{\Qp} L$ for
$L = S/\mathfrak{m}_x$; we then use Galois descent for the action of
$\mathrm{Gal}(L/\Qp)$ to recover a statement about $S$ itself.

\begin{remark}\label{rem:c-example}
Unfortunately, there is no natural extension of Theorem~\ref{thm:rigngbd} to
the Berkovich analytic space $\mathcal{M}(S)$ associated to $S$.
For instance, take $K = \Qp$, $S = \Qp\langle y \rangle$,
and let $M_S$ be free of rank 2 with the action of
$\varphi$ given by the matrix
\[
\begin{pmatrix} 0 & 1 \\ 1 & y/p \end{pmatrix}
\]
(in which $T$ does not appear).
The locus of $x \in M(S)$
where $M_x$ is \'etale is precisely the disc $|y| \leq |p|$,
which does not correspond to an open subset of $\mathcal{M}$.

On the other hand, it may still be the case that $M_S$ is \'etale
if and only if $M_x$ is \'etale (in an appropriate sense) for each
$x \in \mathcal{M}(S)$.
\end{remark}

\begin{remark}
It should be possible to generalize Berger's construction in \cite{LB06} to families of filtered $(\phi, N)$-modules.
With such a generalization,
one would deduce immediately from Theorem~\ref{thm:rigngbd}
that any family of weakly admissible $(\phi, N)$-modules over an affinoid base
(with trivial $\phi$-action on the base)
arises from a Galois representation in a neighborhood of any given
rigid analytic point. However, in view of Remark~\ref{rem:c-example},
we cannot make the  corresponding assertion for Berkovich points.
\end{remark}

\begin{remark}
The families of $\m$-modules considered here are ``arithmetic'' in the sense
that $\varphi$ acts trivially on the base $S$. They correspond to
``arithmetic'' families of Galois representations, such as the
$p$-adic families arising in the theory of $p$-adic modular forms.
There is also a theory of ``geometric'' families of $\m$-modules,
in which $\varphi$ acts as a  Frobenius lift on the base $S$.
These correspond to representations of arithmetic fundamental groups
via the work of Faltings, Andreatta, Brinon, Iovita, et al.
In the latter theory, one does expect the \'etale locus to be open, as in
Hartl's work \cite[Theorem~5.2]{H06}.
One also expects that a family of $\m$-modules is globally \'etale
if and only if it is \'etale over each Berkovich point (but not if it is
only \'etale over each rigid point, as shown by the Rapoport-Zink spaces).
We hope to consider this question in subsequent work.
\end{remark}

\end{document}